\documentclass[11pt,reqno]{amsart}
\usepackage[utf8]{inputenc}
\usepackage{amsmath,amssymb,amsthm,amsfonts}
\usepackage{mathrsfs}
\usepackage{amsfonts}
\usepackage{a4wide}
\newtheorem*{unnumberedthm}{Theorem A}
\allowdisplaybreaks \numberwithin{equation}{section}

\allowdisplaybreaks[1] 
\usepackage{fancyhdr} 
\allowdisplaybreaks
\numberwithin{equation}{section}

\usepackage[numbers,sort&compress]{natbib}
\usepackage{color}
\usepackage[colorlinks=true]{hyperref}

\usepackage{hyperref}
\hypersetup{colorlinks=true,
}
\newtheorem{theorem}{Theorem}[section]

\newtheorem{corollary}[theorem]{Corollary}
\newtheorem{lemma}[theorem]{Lemma}

\theoremstyle{definition}

\newtheorem{remark}[theorem]{Remark}

\newcommand{\R}{\mathbb{R}}
\newcommand{\ds}{\displaystyle}

\newcommand{\myfontsize}{\fontsize{8pt}{11pt}\selectfont}

\begin{document}
	%
	%
	%
	%
	%
	%
	\thanks{The research is  supported by the NSFC (12061012). This paper has been submitted to JDE in Jun. 2024. }
	\title
	{Existence and multiplicity of positive solutions  to a critical elliptic  equation with logarithmic perturbation }
	\maketitle
	\begin{center}
		
		
		\author{Qihan He}
		\footnote{\myfontsize Email addresses:heqihan277@gxu.edu.cn;College of Mathematics and Information Science,\&
			Center for Applied Mathematics of Guangxi (Guangxi University), Guangxi University, Nanning, Guangxi,  530003, P. R. China
		}
		and
		\author{Yiqing Pan}
		\footnote{\myfontsize  Email addresses:13718049940@163.com;(1).College of Mathematics and Information Science, Guangxi University, Nanning, 530003, P. R. China;(2).College of Science, Beibu Gulf University, Qinzhou, 535011, P. R. China}
		
	\end{center}
	
	\begin{abstract} We consider the existence and multiplicity of  positive solutions for the following  critical problem with logarithmic term:
		\begin{equation*}
		\begin{cases}
		-\Delta u={\mu\left|u\right|}^{{2}^{\ast }-2}u+\nu |u|^{q-2}u+\lambda u+\theta u\log {u}^{2}, & \text{ }x\in{\Omega },\\
		\quad \;\:\, u=0,&\text{ } x\in{\partial \Omega },
		\end{cases}
		\end{equation*}
		where $\Omega$ $\subset$ $\R^N$ is a bounded smooth domain, $ \nu, \lambda\in \R$, $\mu>0, \theta<0$, $N\ge3$,
		${2}^{\ast }=\frac{2N}{N-2}$ is the critical  Sobolev exponent for the embedding $H^1_{0}(\Omega)\hookrightarrow L^{2^\ast}(\Omega)$ and $q\in (2, 2^*)$, and which can be seen as a Br$\acute{e}$zis-Nirenberg problem. Under some assumptions on the $\mu, \nu, \lambda, \theta$ and $q$,  we will prove that the above problem has at least two positive solutions: One is the least energy solution, and the other one is the Mountain pass solution.  As far as we know, the existing results on the existence of  positive solutions to a  Br\'ezis-Nirenberg problem are to find a positive solution,  and no one has  given the existence of
		at least two positive solutions on it. So our results is totally new on this aspect.
	\end{abstract}
	\section{Introduction}
	
	\indent In this paper, we investigate the existence and  multiplicity  of positive solutions for the critical problem with logarithmic term:
	\begin{equation}\label{1.1}
	\begin{cases}
	-\Delta u={\mu\left|u\right|}^{{2}^{\ast }-2}u+\nu |u|^{q-2}u+\lambda u+\theta u\log {u}^{2}, & \text{ }x\in{\Omega },\\
	\quad \;\:\, u=0,&\text{ } x\in{\partial \Omega },
	\end{cases}
	\end{equation}
	where $\Omega$ $\subset$ $\R^N$ is a bounded smooth domain, $ \nu, \lambda, \theta\in \R$, $\mu>0$, $N\ge3$,
	${2}^{\ast }=\frac{2N}{N-2}$ is the critical Sobolev exponent for the embedding $H^1_{0}(\Omega)\hookrightarrow L^{2^\ast}(\Omega)$ and $H^1_{0}(\Omega)$ is the closure of $C^\infty_0(\Omega)$ under the norm $\left\|u\right\|:=(\int_\Omega |\nabla u|^2)^\frac{1}{2}$.

	Our motivation for studying \eqref{1.1} comes from  that it originates from  some variational problems in geometry and physics(see \cite{Yamabe,Aubin,bre1,tau1,tau2,uhl,BahriA}),  where lack of compactness also occurs. The most notorious example is Yamabe's problem: find a function $u$ satisfying
	$$
	\begin{cases}
	-4\frac{N-1}{N-2}\Delta u=R^\prime{\left|u\right|}^{{2}^{\ast }-2}u-R(x) u & \text{ }~\hbox{on}~{M },\\
	\quad \;\:\,\quad \;\:\,\;\:\, u>0&\text{ }~\hbox{on}~{M },
	\end{cases}
	$$
	where  $R^\prime$ is some constant, $M$ is  an $N-$dimensional Riemannian manifold, $\Delta$ denotes the Laplacian and $R(x)$ represents the scalar curvature.  As we know,  the presence of the  critical term $|u|^{2^*-2} u$ brings some difficulties to the problem and makes this problem much more interesting.  On the other hand, the problem \eqref{1.1} is also related to the following time-dependent Schr\"{o}dinger equation  
	\begin{equation}\label{e1.2}
	-i\partial_t \Phi=\Delta\Phi+\alpha\Phi\log|\Phi|^2+F(|\Phi|^2)\Phi,   \quad (t,x)\in \R^+\times \Omega \subset \R^+\times\mathbb{R}^N,
	\end{equation}
	where $\alpha\in\R$,$2<p\le 2^*$ and $F$ is a given function, and which  plays a very important role in  nuclear physics, transport and diffusion phenomena, and  quantum mechanics. One can refer  to   \cite{Alfaro,Caze,Shuai,Troy,Wang,Zlosh,Peng} and the references therein for more information on this issue.
	The presence of logarithmic  terms in partial differential equations  has attracted  many interest in recent years, since the  logarithmic  term is sign-changing in $\R^+$ and sub-linear growth near $0$, which makes this problem much more interesting.
	
	
	
	\indent When $\nu=\lambda=\theta=0$ and $\mu=1$, the equation \eqref{1.1} turns into the following equation
	\begin{equation}\label{poh}
	\begin{cases}
	-\Delta u={\left|u\right|}^{{2}^{\ast }-2}u,& \text{ }x\in{\Omega },\\
	\quad \;\:\, u=0,&\text{ } x\in{\partial \Omega }.
	\end{cases}
	\end{equation}
	As we know, the solvability of this problem \eqref{poh} depends heavily on the geometry and topology of the domain $\Omega$. Pohozaev \cite{poh} showed  the first result to problem \eqref{poh}: If the bounded domain $\Omega$ is star-shaped,
	then problem \eqref{poh} has no nontrivial solutions. Some other results about \eqref{poh} can be seen in \cite{Kazdan,Coron,Bahri,Passaseo1,Passaseo2} and the references therein. However, as Br\'{e}zis and Nirenberg alleged in \cite{bre2}, a lower-order terms, such as $\lambda u$, also can reverse this circumstance.
	Br\'{e}zis and Nirenberg \cite{bre2} considered the following classical problem
	\begin{equation}\label{1.3}
	\begin{cases}
	-\Delta u={\left|u\right|}^{{2}^{\ast }-2}u+\lambda u,& \text{ }x\in{\Omega },\\
	\quad \;\:\, u=0,&\text{ } x\in{\partial \Omega },
	\end{cases}
	\end{equation}
	and asserted that the existence of a solution depends heavily on the parameter  $\lambda$ and the dimension $N$. They  showed  that:
	$(i)$ when $N \ge 4$ and  $\lambda\in \left(0,\lambda_{1}(\Omega)\right)$, there exists a positive sulution for  \eqref{1.3};
	$(ii)$ when $N=3$ and $\Omega$ is a ball, problem  \eqref{1.3} has a positive solution if and only if  $\lambda\in \left(\frac{1}{4}\lambda_{1}(\Omega),\lambda_{1}(\Omega)\right)$;
	$(iii)$~problem  \eqref{1.3} has no solutions when $\lambda<0$ and $\Omega$ is star-shaped, where $\lambda_1(\Omega)$ denotes the first eigenvalue of $-\Delta $ with zero Dirichlet boundary value.
	In \cite{Deng}, Deng et al. investigated the existence and non-existence of positive solutions for the following  problem with critical exponent and logarithmic perturbation
	\begin{equation}\label{Dend1.1}
	\begin{cases}
	-\Delta u={\left|u\right|}^{{2}^{\ast }-2}u+\lambda u+\theta u\log {u}^{2}, & \text{ }x\in{\Omega },\\
	\quad \;\:\,  u=0,&\text{ } x\in{\partial \Omega },
	\end{cases}
	\end{equation}
	which corresponds to \eqref{1.1} with $\mu=1$ and $\nu=0$ and  where the logarithmic term  $u\log {u}^{2}$ brings us two interesting  facts:\,\: One is that
	compared with $|u|^{2^*-2}u$, $u\log u^2$ is also  a lower-order term at infinity, and the other one is that $u\log {u}^{2}$ is sub-linear growth near $0$ and sign-changing in $(0, +\infty)$.
	They proved that problem \eqref{Dend1.1} admits a positive ground state solution, which is also a Mountain pass type solution when $N \ge4, \lambda\in\R$ and $\theta>0$. For the case of $\theta<0$, they only got a positive solution for $(\lambda, \theta )\in B_0\cup C_0$,  and $N=3$, or $N=4$ and $\frac{32 e^\frac{\lambda}{\theta }}{\rho_{max}^2}< 1$,
	where $B_0:=\{(\lambda, \theta )|\lambda \in [0,\lambda_1(\Omega)),  \theta <0,\frac{1}{N}(\frac{\lambda_1(\Omega)-\lambda}{\lambda_1(\Omega)})^\frac{N}{2}S^\frac{N}{2}+\frac{\theta }{2}|\Omega|>0\},$
	$C_0:=\{ (\lambda, \theta )|\lambda \in \R,  \theta <0, \frac{1}{N}S^\frac{N}{2}+\frac{\theta }{2}e^{-\frac{\lambda}{\theta  }}|\Omega|>0\}$
	and $\rho_{max}:=\sup\{r\in (0, +\infty): B(0,r)\subset \Omega\}.$ Besides, they also gave some non-existence  results   under some suitable assumptions on $\lambda$ and $\theta $.  The readers can refer to \cite{Deng} for more details. We want to emphasize the following two facts: (1) When  $(\lambda, \theta )\in B_0\cup C_0$, they have found that the ground state  energy is smaller than the Mountain pass energy; (2) Even though they  mainly used the Mountain pass theorem to show the existence of positive solution, they didn't  know whether the solution, they obtained,  is a Mountain pass solution  or not, and left it as an open problem.  Liu and Zou \cite{liut} showed  the existence of sign-changing solution with exactly two nodal domain to \eqref{Dend1.1} for $ \lambda \in \mathbb{R}, \theta>0$ and any smooth bounded domain $\Omega$. Particularly,  when $\Omega=B(0,R)$ is a ball, they  constructed infinitely many radial sign-changing solutions with alternating signs and prescribed nodal characteristic.
	Recently, Zou and his collaborators \cite{hlz,hlsz} considered the problem \eqref{1.1} with $\nu=0$, i.e,
	\begin{equation}\label{zou1.2}
	\begin{cases}
	-\Delta u=\mu{\left|u\right|}^{{2}^{\ast }-2}u+\lambda u+\theta u\log {u}^{2}, & \text{ }x\in{\Omega },\\
	\quad \;\:\,  u=0,&\text{ } x\in{\partial \Omega },
	\end{cases}
	\end{equation}
	which is more general than \eqref{Dend1.1} and showed the following results:
	\begin{unnumberedthm}\label{thmA}
		
		(1)~~If $N\geq 4$ and $(\lambda,\mu,\theta)\in M_1\cup M_2$, then problem \eqref{zou1.2} has a positive local minimum solution $\bar{u}$  and a positive ground state  solution $\tilde{u}$ such  that $J(\bar{u})=\tilde{c}_\rho<0$ and $J(\tilde{u})=\tilde{c}_g<0$, where
		\begin{equation}\label{m1}
		M_1:=\{(\lambda, \mu, \theta)|\lambda \in [0,\lambda_1(\Omega)),  \mu>0,\theta<0,\frac{\mu}{N}(\frac{\lambda_1(\Omega)-\lambda}{\mu\lambda_1(\Omega)})^\frac{N}{2}S^\frac{N}{2}+\frac{\theta}{2}|\Omega|>0\},
		\end{equation}
		\begin{equation}\label{m2}
		M_2:=\{ (\lambda, \mu, \theta)|\lambda \in \R,  \mu>0,\theta<0, \frac{1}{N}\frac{1}{{\mu}^\frac{N-2}{2}}S^\frac{N}{2}+\frac{\theta}{2}e^{-\frac{\lambda}{\theta }}|\Omega|>0\},
		\end{equation}
		$$ J(u)=\frac{1}{2}\int_{\Omega}\left|\nabla u\right|^{2}-\frac{\mu}{2^\ast}\int\left|u_+\right|^{2^\ast}
		-\frac{\theta}{2}\int{(u_{+})^2}(\log (u_{+})^2+\frac{\lambda}{\theta}-1),$$
		$$\tilde{c}_g:=\inf\limits_{u\in \{v\in H^1_0(\Omega):J^\prime(v)=0\}}J(u),$$
		$$\tilde{c}_\rho:=\inf\limits_{|\nabla u|_2<\rho}J(u)$$
		and $\rho$ is a suitable constant.
		
		(2) If $N\geq 4$, $(\lambda,\mu,\theta)\in M_1\cup M_2$ and $|\theta|e^{\frac{N}{2}-\frac{\lambda}{\theta}-1}<\rho^2$, then $\tilde{c}_\rho=\tilde{c}_g$, which implies that the solution $\bar{u}$  is also a positive ground state solution.
	\end{unnumberedthm}
	Besides, we want to  mention that Han and his collaborators \cite{lhw, zhw} generalized the results of \cite{Deng,hlz,hlsz} to the bi-harmonic critical problem with  logarithmic perturbation.
	In \cite{Deng}, that the Nehari manifold is a natural manifold can not be proved, which causes that   we can not find a positive ground state solution to \eqref{Dend1.1} for the case of $\theta<0$. Thus we think that Zou and his collaborators gave an important method to find a positive ground state solution for some problems, whose Nehari manifold is not  a natural manifold.
	Since problem \eqref{Dend1.1} is a special case of \eqref{zou1.2}, it follows  from the results of \cite{hlz,hlsz}  that problem \eqref{Dend1.1} has a positive ground state solution for $N\geq 4$ and $(\lambda, \theta )\in B_0\cup C_0$. On the other hand, when $(\lambda,\mu,\theta)\in M_1\cup M_2$, $J(u)$ has   Mountain pass geometry structure, which leads to  a natural question whether problem \eqref{zou1.2} has a   Mountain pass solution. For this question,   Zou and his collaborators \cite{hlsz}   proposed a
	conjecture: When $N=4$ and $(\lambda,\mu,\theta)\in M_1\cup M_2$,  problem \eqref{zou1.2} possesses  a positive Mountain pass solution at level $\tilde{c}_M>0$, where
	$$\tilde{c}_M:=\inf\limits_{\gamma\in \tilde{\Gamma} }\sup\limits_{t\in [0,1]}J(\gamma(t)),$$
	$$\tilde{\Gamma}:=\{\gamma\in C([0,1], H^1_0(\Omega)):\gamma(0)=\bar{u}, J(\gamma(1))<J(\bar{u})\},$$
	and $\bar{u}$ is the local minimum solution given by Theorem A.
	
	Inspired by the above results, we want to study the existence and multiplicity of positive solutions to \eqref{1.1} and give an affirmative answer to the  conjecture, proposed by Zou and his collaborators in \cite{hlsz}.

	To find the positive solution of \eqref{1.1}, we define a functional as bellows:
	\begin{equation}\label{fun1}
	I_\nu(u)=\frac{1}{2}\int_{\Omega}\left|\nabla u\right|^{2}-\frac{\mu}{2^\ast}\int\left|u_+\right|^{2^\ast}-\frac{\nu}{q}\int\left|u_+\right|^{q}
	-\frac{\lambda}{2}\int |u_+|^2-\frac{\theta}{2}\int{(u_{+})^2}(\log (u_{+})^2-1),~u\in H^1_0(\Omega),
	\end{equation}
	
	or
	\begin{equation}\label{fun2}
	I_\nu(u)=\frac{1}{2}\int_{\Omega}\left|\nabla u\right|^{2}-\frac{\mu}{2^\ast}\int\left|u_+\right|^{2^\ast}-\frac{\nu}{q}\int\left|u_+\right|^{q}
	-\frac{\theta}{2}\int{(u_{+})^2}(\log (u_{+})^2+\frac{\lambda}{\theta}-1),~u\in H^1_0(\Omega),
	\end{equation}%
	where $u_{+}=max\{u,0\},\,u_{-}=-max\{-u,0\}$. It is easy to see that $I_\nu(u)$ is well-defined in $H_0^1(\Omega)$ and every nonnegative critical point of $I_\nu(u)$ corresponds to a solution of \eqref{1.1}.
	
	Before stating our results, we introduce some notations.
	Hereafter, we use $\int$ to denote $\int_\Omega~\mathrm{d}x$, unless specifically stated,    and let $S$ and $\lambda_1(\Omega)$ be the best Sobolev constant of the embedding  $H^1(\R^N) \hookrightarrow L^{2^*}(\R^N)$ and the first eigenvalue of $-\Delta $ with zero Dirichlet boundary value, respectively,  i.e,
	$$S:=\inf\limits_{u\in H^1(\R^N)\setminus\{0\}}\frac{\int_{\R^N}|\nabla u|^2~\mathrm{d}x}{(\int_{\R^N}|u|^{2^*}~\mathrm{d}x)^\frac{2}{2^*}}$$
	and
	$$\lambda_1(\Omega):=\inf\limits_{u\in H^1_0(\Omega)\setminus\{0\}}\frac{\int_{\Omega}|\nabla u|^2~\mathrm{d}x}{\int_{\Omega}|u|^2~\mathrm{d}x}.$$
	
	Let
	\begin{equation}\label{m3}
	M_3:=\{ (\lambda, \mu, \theta)|\lambda \in [0,\lambda_1(\Omega)),  \mu>0,\theta<0, \frac{1}{N}(\frac{\lambda_1(\Omega)-\lambda}{\lambda_1(\Omega)})^\frac{N}{2}(\frac{q}{q\mu+2^*\nu})^\frac{N-2}{2}S^\frac{N}{2}+\frac{\theta }{2}e^{-\frac{2\nu}{q\theta}}|\Omega|>0\},
	\end{equation}
	and
	\begin{equation}\label{m4}
	M_4:=\{ (\lambda, \mu, \theta)|\lambda \in \R,  \mu>0,\theta<0, \frac{1}{N}(\frac{q}{q\mu+2^*\nu})^\frac{N-2}{2}S^\frac{N}{2}+\frac{\theta }{2}e^{-\frac{2\nu}{q\theta}-\frac{\lambda}{\theta}}|\Omega|>0\}.
	\end{equation}

	At the same time, we set
	$$\left\|v\right\|^2:=\int|\nabla v|^2,\,~~v\in {H}_{0}^{1}(\Omega),$$
	\begin{equation}\label{defc1}
	c_{\rho}:=\inf\limits_{u\in A}I_\nu(u),
	\end{equation}
	and
	\begin{equation}\label{defc2}
	c_\eta:=\inf\limits_{u\in \eta}I_\nu(u),
	\end{equation}
	where $A:=\{||u||_2<\rho\}$,  $\rho $ will be  given in  Lemma \ref{Lemma1} and $
	\eta:=\{u\in H_0^1(\Omega):I_\nu'(u)=0\}. $

	Our results can be stated as below:

	\begin{theorem}\label{buth1}
		Assume that $N\geq 3$ and  $q\in (2, 2^*)$. If $\nu\leq 0$ and  $(\lambda,\mu,\theta)\in M_1\cup M_2$ or $\nu>0$ and $(\lambda,\mu,\theta)\in  M_3\cup M_4$,
		then problem \eqref{1.1} has a positive local minimum solution $u_0\in A$ and a positive ground state solution $u_1\in \eta$ such that $I_\nu(u_0)=c_{\rho}<0 $
		and $I_\nu(u_1)=c_\eta<0,$ where $M_1, M_2, M_3, M_4$ and $c_\rho, c_\eta$ are defined in \eqref{m1}, \eqref{m2} and \eqref{m3}-\eqref{defc2},
		respectively.
	\end{theorem}

	\begin{theorem}\label{buth3}
		Under the assumptions of Theorem \ref{buth1}, we have that  $ c_\eta=c_\rho.$
	\end{theorem}
	\begin{remark} For the case of $\nu=0$ and $N\geq 4$,
		the results of Theorems \ref{buth1}   have been proven   in \cite{hlz, hlsz}. At the same time, they also proved that
		$c_\rho =c_\eta $ under the assumptions that $|\theta|e^{\frac{N}{2}-\frac{\lambda}{\theta}-1}<\rho^2.$ Comparing with the results of \cite{hlz, hlsz},
		the results of this paper contain more cases $\nu<0$, $\nu>0$  and $N=3$, and we show $ c_\eta=c_\rho$ without the assumptions that $|\theta|e^{\frac{N}{2}-\frac{\lambda}{\theta}-1}<\rho^2,$ which tells us  that the solution $u_0$, found in Theorem \ref{buth1}, is also a ground state solution.
		
	\end{remark}

	It is easy to see that $ I_\nu$ has  a Mountain pass geometry structure  since Theorem \ref{buth1}  shows the existence of a local minimum point $u_0$ and $I_\nu(t u) \rightarrow-\infty$ as $t \rightarrow +\infty$  for any fixed $u\in H^1_0(\Omega)$ with  $u_{+} \neq 0$. So
	$$
	c_M:=\inf _{\gamma \in \Gamma } \sup _{t \in[0,1]} I_\nu(\gamma(t))
	$$
	is well-defined,
	where
	$$
	\Gamma :=\left\{\gamma \in C\left([0,1], H_0^1(\Omega)\right): \gamma(0)=u_0, I_\nu(\gamma(1))<I_\nu(
	u_0)\right\},
	$$
	and
	$u_0$ is  given in  Theorem \ref{buth1}. Next, we would show that problem \eqref{1.1}  possesses a positive Mountain pass solution at level $c_M>0$.

	\begin{theorem}\label{th4}
		Problem \eqref{1.1} has a positive Mountain pass  solution $u$ such that $I_\nu(u)=c_M>0,$
		provided one of the following assumptions holds:
		
		(1)~~$3\leq N\leq 5$, $\nu\leq 0$, $(\lambda,\mu,\theta)\in M_1\cup M_2 $  and $q\in (2, 2^*-1);$
		
		(2)~~$N\geq 3$, $\nu>0$, $(\lambda,\mu,\theta)\in M_3\cup M_4$ and $q\in (2, 2^*).$
		
	\end{theorem}
	
	\begin{remark}
		Theorem \ref{buth3} is very important for the proof of Theorem \ref{th4}. As we see in Section 4, to complete the proof of Theorem \ref{th4},    we have to show  the Mountain pass energy $c_M<c_\eta+\frac{1}{N}\mu^{-\frac{N-2}{2}}S^\frac{N}{2}$. According to  the definitions of $c_M$ and $\Gamma$, we construct a path $\gamma_0(t)\in \Gamma$ connecting $u_0$  with a function  related to $S$ and prove that $c_M<c_\rho+\frac{1}{N}\mu^{-\frac{N-2}{2}}S^\frac{N}{2}$, which, together with Theorem \ref{buth3}, tells us that $c_M$ is indeed smaller than $c_\eta+\frac{1}{N}\mu^{-\frac{N-2}{2}}S^\frac{N}{2}$.
		
	\end{remark}
	
	\begin{remark}
		It is easy to see that when $\nu=0$, problem \eqref{1.1} is reduced to \eqref{zou1.2}. So, following from our Theorem \ref{th4}, one can see that problem \eqref{zou1.2} has a positive Mountain pass  solution $u$ such that $J(u)=\tilde{c}_M$ for $3\leq N\leq 5$ and  $(\lambda,\mu,\theta)\in M_1\cup M_2, $ which gives an affirmative answer to the  conjecture, proposed by Zou and his collaborators in \cite{hlsz}.
	\end{remark}
	
	\begin{remark}
		We can not find a positive Mountain pass  solution to \eqref{1.1} for $N\geq6,$ $\nu\leq 0$ and $(\lambda,\mu,\theta)\in M_1\cup M_2 $ in Theorem \ref{th4}, since $\int U_n^{2^*-1}=O\Big(\frac{1}{n^{\frac{N-2}{2}}}\Big)\leq O\Big(\frac{1}{n^2}\Big)=\int U_n^2$ for  $N\geq6$(See the proof of Lemma \ref{lm5.3}). Therefore, we have to leave it as an open problem and
		continue to study it in the future.
		
	\end{remark}
	
	\begin{remark}
		In Theorems \ref{buth1}  and \ref{th4}, we require $(\lambda,\mu,\theta)\in M_1\cup M_2 $ for $\nu\leq 0$ and $(\lambda,\mu,\theta)\in M_3\cup M_4$ for $\nu>0$,   just to ensure that the functional $I_\nu(u)$ has a Mountain pass geometry structure. We don't know whether the domains of the parameters are optimal. If one can find some other domains to let  the functional $I_\nu(u)$ has the  Mountain pass geometry structure, then  $ M_1\cup M_2 $ or $M_3\cup M_4$ can be replaced by the corresponding domains.
		
	\end{remark}
	
	Before closing our introduction, we outline the main ideas and the approaches in the proofs of our main results.
	
	We use some  similar methods as \cite{hlz, hlsz} to show Theorem \ref{buth1}, but the presence of $\nu|u|^{q-2}u$ makes the proof much more complicated.
	
	For Theorem \ref{buth3}, by Theorem \ref{buth1} and the definitions of $c_\rho$ and $c_\eta$,  we  firstly have  that $u_0\in \eta$ and
	\begin{equation}\label{inel}
	c_\rho=I_\nu(u_0)\geq c_\eta.
	\end{equation}
	On the other hand, following from Theorem \ref{buth1}, there exists at least a positive solution $u_1$ of \eqref{1.1} such that $I_\nu(u_1)=c_\eta<0.$ Letting $g(t):=I_\nu(tu_1), t\in \R^+$, direct computations imply  that $g(t)$ has at most two extreme points,  and
	the geometry structure of the energy functional $I_\nu$ tells us that
	$g(t)$ has at least two extreme points. So we obtain that $g(t)$ has only two extreme points $t_1, t_2\in \R^+$ with $t_1<t_2$. Using the geometry structure of the energy functional $I_\nu$ and the facts that $u_1$ is a positive solution of \eqref{1.1} and $I_\nu(u_1)<0$ again, we can see that $t_1=1$ is the local minimum point of $g(t)$ and $||tu_1||<\rho$ for any $t\in (0,1]$, which implies that $||u_1||<\rho.$ That is, $u_1\in A.$ Thus
	$c_\eta=I_\nu(u_1)\geq c_\rho,$ which, together with \eqref{inel}, tells us  that $c_\eta=c_\rho.$

	To find a positive Mountain pass solution to \eqref{1.1}, we firstly  prove some technical Lemmas (See Lemmas \ref{lm4.1}, \ref{lm4.2}, \ref{lm4.3}) to conclude the Corollary \ref{c1}, which is very important for us. Based on the Corollary \ref{c1}, we  construct a path $\gamma_0(t)\in  \Gamma $ with the form $\gamma_0(t):=u_0+t T U_n$ such that  $c_M\leq \sup\limits_{t\in [0,1]}I_\nu(\gamma_0(t))<c_\eta+\frac{1}{N}\mu^{-\frac{N-2}{2}}S^\frac{N}{2},$ where $U_n$ will be  defined in \eqref{Un} and $T$ will be  established in Lemma \ref{lm5.1}.
	After getting  $c_M<c_\eta+\frac{1}{N}\mu^{-\frac{N-2}{2}}S^\frac{N}{2},$   it is not difficult to show that any $(PS)_{c_M}$ sequence  $\{v_n\}$ of $I_\nu$  is convergent in $H^1_0(\Omega)$. Therefore, we get the Theorem \ref{th4}.

	The paper is organized as follows: Section 2 is dedicated to some preliminary results and the proofs of Theorems \ref{buth1} and \ref{buth3}. We put some important inequalities and energy estimates
	into the Section 3.
	In Section 4, we prove Theorem \ref{th4}.
	

	\section{Preliminary results and the proofs of Theorems \ref{buth1} and \ref{buth3}}
	
	\begin{lemma}\label{Lemma1}
		Assume that $N\geq 3$ and $q\in (2, 2^*)$.     Then we see that   the funtional $I_\nu(u)$ has Mountain pass geometry struture:~
		$(i) $~~~there exist $\alpha$, $\rho$$>0$ such that $I_\nu(v)\ge\alpha$ for all $\left\|v\right\|=\rho$;~
		$(ii)$~~~there exists $\omega$$\in$$H_{0}^{1}(\Omega)$ such that $\left\|\omega\right\|$$\ge$$\rho$ and $I_\nu(\omega)<0$,
		provided  one of the following assumptions holds:

		(I)~ $\nu\leq 0$ and $(\lambda,\mu,\theta)\in M_1\cup M_2;$
		
		
		(II)~~$\nu>0$ and $(\lambda,\mu,\theta)\in M_3 \cup M_4.$

	\end{lemma}
	
	\begin{proof} Since $\mu>0$, as the proof of Lemma 3.1 of  \cite{Deng},  we can find a $\varphi\in H^1_0(\Omega)$ and   $t_{0}\in \R^+$ large enough such that
		\begin{align*}
		I_\nu(t_{0}\varphi)<0 \quad and\quad \left\|t_{0}\varphi \right\|>\rho .
		\end{align*}
		So (ii) is true for any Cases (I) and (II).
		
		We divide the proof of (i) into three Cases.

		\textbf{Case 1:} $\nu\leq 0$ and $(\lambda,\mu,\theta)\in M_1\cup M_2  $
		
		If $\nu\leq 0$, then we see that $I_\nu(u)\geq I_0(u)$ for any $u\in H^1_0(\Omega).$  According to the Lemma 3.1 of  \cite{Deng}, we get the conclusion (i) for Case (I).

		\textbf{ Case 2:} $\nu>0$ and  $(\lambda,\mu,\theta)\in M_3$
		
		Since  $\theta<0$, we have
		\begin{align*}
		& -\frac{\nu}{q}\int u^2_{+}-\frac{\theta }{2}\int u^2_{+}(\log u^2_{+}-1)\\
		&=-\frac{\theta }{2}\int u^2_{+}\log(e^{\frac{2\nu}{q\theta}-1} u^2_{+})\\
		&=-\frac{\theta }{2}e^{1-\frac{2\nu}{q\theta}}\int _{\{e^{\frac{2\nu}{q\theta}-1} u^2_{+}\geq 1\}}e^{\frac{2\nu}{q\theta}-1}u^2_{+}\log (e^{\frac{2\nu}{q\theta}-1}u^2_{+}) -\frac{\theta }{2}e^{1-\frac{2\nu}{q\theta}}\int _{\{e^{\frac{2\nu}{q\theta}-1} u^2_{+}\leq 1 \}}e^{\frac{2\nu}{q\theta}-1}u^2_{+}\log (e^{\frac{2\nu}{q\theta}-1} u^2_{+}) \\
		&\ge -\frac{\theta }{2}e^{1-\frac{2\nu}{q\theta}}\int _{\{e^{\frac{2\nu}{q\theta}-1} u^2_{+}\leq 1 \}}e^{\frac{2\nu}{q\theta}-1}u^2_{+}\log (e^{\frac{2\nu}{q\theta}-1} u^2_{+}) \\
		&\ge   \frac{\theta }{2}e^{-\frac{2\nu}{q\theta}}|\Omega|.
		\end{align*}
		So  we have
		\begin{equation}\label{bu5.71}
		I_\nu(u)\geq \frac{1}{2}\frac{\lambda_1(\Omega)-\lambda}{\lambda_1(\Omega)}\int |\nabla u|^2-(\frac{\mu}{2^*}+\frac{\nu}{q})S^{-\frac{2^*}{2}}(\int |\nabla u|^2)^\frac{2^*}{2}
		+\frac{\theta }{2}e^{-\frac{2\nu}{q\theta}}|\Omega|,
		\end{equation}
		where we have used $|s|^q\leq |s|^{2^*}+|s|^2.$
		We let $\alpha :=\frac{1}{N}(\frac{\lambda_1(\Omega)-\lambda}{\lambda_1(\Omega)})^\frac{N}{2}(\frac{q}{q\mu+2^*\nu})^\frac{N-2}{2}S^\frac{N}{2}+\frac{\theta }{2}e^{-\frac{2\nu}{q\theta}}|\Omega|$ and $\rho :=(\frac{\lambda_1(\Omega)-\lambda}{\lambda_1(\Omega)})^\frac{N-2}{4}(\frac{q}{q\mu+2^*\nu})^\frac{N-2}{4}S^\frac{N}{4}$. Since $(\lambda, \mu,\theta)\in M_3$, we have $ \alpha >0$ and $\rho >0$. Following from \eqref{bu5.71}, we have that,  for any $||v||=\rho $,
		\begin{align*}
		I_\nu(v)&\geq \frac{1}{2}\frac{\lambda_1(\Omega)-\lambda}{\lambda_1(\Omega)}\rho ^2-(\frac{\mu}{2^*}+\frac{\nu}{q})S^{-\frac{2^*}{2}}\rho ^{2^*}
		+\frac{\theta }{2}e^{-\frac{2\nu}{q\theta}}|\Omega|\\
		&=\frac{1}{N}(\frac{\lambda_1(\Omega)-\lambda}{\lambda_1(\Omega)})^\frac{N}{2}(\frac{q}{q\mu+2^*\nu})^\frac{N-2}{2}S^\frac{N}{2}+\frac{\theta }{2}e^{-\frac{2\nu}{q\theta}}|\Omega|\\
		&=\alpha >0.
		\end{align*}
		\textbf{Case 3:} $\nu>0$ and $(\lambda,\mu,\theta)\in M_4$
		
		Since  $\theta <0$, we have
		\begin{align*}
		& -\frac{\nu}{q}\int u^2_{+}-\frac{\lambda}{2}\int u_{+}^2-\frac{\theta }{2}\int u^2_{+}(\log u^2_{+}-1)\\
		&=-\frac{\theta }{2}\int u^2_{+}\log(e^{\frac{2\nu}{q\theta}+\frac{\lambda}{\theta}-1} u^2_{+})\\
		&\ge-\frac{\theta }{2}\int _{\{e^{\frac{2\nu}{q\theta}+\frac{\lambda}{\theta}-1} u^2_{+}\leq 1 \}}u^2_{+}\log (e^{\frac{2\nu}{q\theta}+\frac{\lambda}{\theta}-1} u^2_{+})\\
		&\ge -\frac{\theta }{2}e^{1-\frac{2\nu}{q\theta}-\frac{\lambda}{\theta}}\int _{\{e^{\frac{\lambda}{\theta  }-1} u^2_{+}\leq 1 \}}-e^{-1}~ \mathrm{d}x\\
		&\ge \frac{\theta }{2}e^{-\frac{2\nu}{q\theta}-\frac{\lambda}{\theta}}|\Omega|.
		\end{align*}
		So  we have
		\begin{equation}\label{bubu5.71}
		I_\nu(u)\geq \frac{1}{2}\int |\nabla u|^2-(\frac{\mu}{2^*}+\frac{\nu}{q})S^{-\frac{2^*}{2}}(\int |\nabla u|^2)^\frac{2^*}{2}
		+\frac{\theta }{2}e^{-\frac{2\nu}{q\theta}-\frac{\lambda}{\theta}}|\Omega|.
		\end{equation}
		We let $\alpha:=\frac{1}{N}(\frac{q}{q\mu+2^*\nu})^\frac{N-2}{2}S^\frac{N}{2}+\frac{\theta }{2}e^{-\frac{2\nu}{q\theta}-\frac{\lambda}{\theta}}|\Omega|$ and $ \rho :=(\frac{q}{q\mu+2^*\nu})^\frac{N-2}{4}S^\frac{N}{4}$. Since $(\lambda, \mu,  \theta )\in M_4$, we have $ \alpha>0$ and $\rho>0$. Following from \eqref{bubu5.71}, we have that,  for any $||v||=\rho$,
		$$I_\nu(v)\geq \frac{1}{2} \rho^2-(\frac{\mu}{2^*}+\frac{\nu}{q})S^{-\frac{2^*}{2}}\rho^{2^*}
		+\frac{\theta }{2}e^{-\frac{2\nu}{q\theta}-\frac{\lambda}{\theta}}|\Omega|=\frac{1}{N}(\frac{q}{q\mu+2^*\nu})^\frac{N-2}{2}S^\frac{N}{2}+\frac{\theta }{2}e^{-\frac{2\nu}{q\theta}-\frac{\lambda}{\theta}}|\Omega|=\alpha>0.$$

		We complete the proof.
	\end{proof}
	
	
	\begin{lemma}\label{Lemma1-1} Under the hypotheses of Lemma \ref{Lemma1}, Let $c_\rho$ be given by Theorm \ref{buth1}. Then $-\infty<c_\rho<0.$
	\end{lemma}
	\begin{proof}
		For any $u\in A:=\{||u||_2<\rho\}$,
		direct computations imply that
		$$I_\nu(u)\geq -C_1||u||^{2^*}-C_2||u||^q-C_2|\Omega|>-C_1\rho^{ 2^* }-C_2\rho^q-C_2|\Omega|>-\infty,$$
		which tells us that
		$c_\rho>-\infty.$
		
		Choosing $u\in H_0^1(\Omega)$ with $u>0$  in $\Omega$, then we have that, for $t>0$ small enough,
		\begin{align*}
		I_\nu(tu)&=t^2\Big[\frac{1}{2}\int_{\Omega}\left|\nabla u\right|^{2}-\frac{\lambda}{2}\int |u_+|^2-\frac{\nu}{q}t^{q-2}\int\left|u_+\right|^{q}\\
		&-\frac{\mu}{2^\ast}t^{2^\ast-2}\int\left|u_+\right|^{2^\ast}-\frac{\theta}{2}\int{(u_{+})^2}(\log (u_{+})^2-1)-\frac{\theta}{2}\log t^2\int{(u_{+})^2}\Big]\\
		&<0.
		\end{align*}
		So we can choose $t_0>0$ small enough such that $t_0u\in A$ and $I_\nu(t_0u)<0$, which implies that
		$$c_\rho\leq I_\nu(t_0u)<0.$$
		
		We complete the proof.
	\end{proof}

	\begin{lemma}\label{Lemma2} Assume that $N\geq 3$ and $\theta<0$.  Then
		any  $(PS)_{c}$ sequence $\{u_{n}\}$ must be    bounded   in $H^1_{0}(\Omega)$, where $c\in \R$.
	\end{lemma}
	
	\begin{proof}
		By the definition of $(PS)_c$ sequence and the facts that
		$s^2log s^2\ge -e^{-1}, \theta<0$ and for any $C_1\in \R, $
		\begin{equation}\label{h2.5}
		\begin{split}
		& C_1\int(u_{n})^2_{+}
		+ \int(u_{n})^2_{+}\log(u_{n})^2_{+}
		=e^{-C_1}\int e^{C_1}(u_{n})^2_{+}\log \big(e^{C_1}(u_{n})^2_{+}\big)
		\geq  -e^{-1-C_1}|\Omega|,
		\end{split}
		\end{equation}
		we have that, for $n$ large enough,
		\begin{equation}\label{3291}
		\begin{split}
		c+||u_n||+1&\ge I_\nu(u_n)-\frac{1}{q}<I'(u_n),u_n>\\
		&\geq\frac{q-2}{2q}\left\|u_{n}\right\|^2-\frac{(q-2)\lambda-q\theta}{2q}\int(u_{n})^2_{+}
		-\theta\frac{q-2}{2q}\int(u_{n})^2_{+}\log(u_{n})^2_{+}\\
		&\geq\frac{q-2}{2q}\left\|u_{n}\right\|^2-C|\Omega|,\\
		\end{split}
		\end{equation}
		which implies that  $\{u_n\}$ is bounded  in $H_0^1(\Omega).$
	\end{proof}

	\begin{lemma}\label{Lemma3} If  $c_\rho$ is attained   and    $\theta<0$, then  $-\infty<c_\eta<0.$
		
	\end{lemma}
	
	\begin{proof}
		Assume that $c_\rho$ is attained by $u_0$. We see that $u_0>0$, $I_\nu^\prime(u_0)=0$ and $I_\nu(u_0)=c_\rho<0$, which implies that
		$c_\eta\le I_\nu(u_0)=c_\rho<0.$ So $c_\eta<0.$

		Following  from $\theta<0$ and \eqref{3291}, we have that, for any $u\in \eta:= \{u\in H^1_0(\Omega):~I_\nu^\prime(u)=0\}$,
		\begin{equation}
		\begin{split}
		I_\nu(u)&= I_\nu(u)-\frac{1}{q}<I'(u),u>\geq\frac{q-2}{2q}\left\|u\right\|^2-C|\Omega|\geq -C|\Omega|>-\infty,
		\end{split}
		\end{equation}
		which implies that
		$c_\eta>-\infty.$
	\end{proof}

	\begin{lemma}\label{lmp3}
		Assume that $I_\nu^\prime(u)=0$ with $u>0$ and set $g(t):=I_\nu(tu), t>0$. Then $g(t)$ has at most two extreme points.
	\end{lemma}
	
	\begin{proof}
		By direct computations, we have
		\begin{equation*}
		\begin{split}
		g^\prime(t)&=t\Big(|\nabla u |_2^2-\lambda |u |_2^2-\theta \int u ^2\ln u ^2-\nu\int|u|^{q}t^{q-2}-\mu\int|u |^{2^*}t^{2^*-2}-\theta \int u ^2\ln t^2\Big)\\
		&=t\Big[\mu\int|u |^{2^*}(1-t^{2^*-2})+\nu\int|u |^{q}(1-t^{q-2})-\theta \int u ^2\ln t^2\Big]=:g_1(t)t,
		\end{split}
		\end{equation*}
		
		$$g_1^\prime(t)=t^{q-3}[-2\theta \int u ^2t^{2-q}-(2^*-2)\mu\int|u |^{2^*} t^{2^*-q}-(q-2)\nu\int|u |^{q} ]=:g_2(t)t^{q-3}$$
		
		and
		$$g_2^\prime(t)=-2\theta(2-q) \int u ^2t^{1-q}-(2^*-2)(2^*-q)\mu\int|u |^{2^*} t^{2^*-q-1}<0,$$
		which, together with  $\mu>0$ and $\theta<0$, implies  that $g_1^\prime (t)=0$ has a unique root $t_3$, and $g_1^\prime (t)>0$ in $(0, t_3)$ and  $g_1^\prime (t)<0$ in $(  t_3, +\infty)$. Following from  $\lim\limits_{t\to 0^+}g_1(t)=\lim\limits_{t\to +\infty}g_1(t)=-\infty,$
		we have  that $g_1(t)=0$ has at most two roots.  So we have that $g^\prime(t)=0$ has at most two roots, i,e,  $g(t)$ has at most two extreme points.
	\end{proof}

	\begin{proof}[\textbf{The proof on $c_\rho$ of  Theorem \ref{buth1}:}]
		By Lemma \ref{Lemma1} and Lemma \ref{Lemma1-1}, we can take a minimizing
		sequence $\{u_n\}$ for $c_\rho $ with $|\nabla u_n|_2<\rho-\varepsilon$ and $\varepsilon>0$ small enough. By  Ekeland's variational principle(see \cite{Struwe}), we known that there exists a sequence $\{\tilde{u}_n\}\subset H^1_0(\Omega)$ such that $||u_n-\tilde{u}_n||\to 0$, $I_\nu(\tilde{u}_n)\to c_\rho$  and $I_\nu'(\tilde{u}_n)\to 0$. By Lemma \ref{Lemma2}, we can see that $\{\tilde{u}_n\}$ is bounded in $H_0^1(\Omega)$. Hence, passing to a subsequence, we may   assume that
		$$
		\begin{aligned}
		&\tilde{u}_n \rightharpoonup u \text { weakly in } H_0^1(\Omega),\\
		& \tilde{u}_n  \rightharpoonup u \text { weakly in } L^{2^*}(\Omega), \\
		& \tilde{u}_n  \rightarrow u \text { strongly in } L^m(\Omega) \text { for } 1\leq m<2^*, \\
		& \tilde{u}_n  \rightarrow u \text { almost everywhere in } \Omega .
		\end{aligned}
		$$
		By the weak lower semi-continuity of the norm, we see that $|\nabla u|_2\leq \rho-\varepsilon<\rho$. By $I_\nu^\prime(\tilde{u}_n)\to 0$, we have that  $I_\nu'(u)=0$ and $u\ge 0$.
		Letting  $\omega_n=\tilde{u}_n-u$, one gets that
		\begin{equation}\label{A.2}
		\begin{split}
		\int_\Omega|\nabla\omega_n|^2&=\mu|\omega^+_n|_{2^*}^{2^*}+o_n(1),\\
		I_\nu(\tilde{u}_n)=I_\nu(u)&+\frac{1}{N}\int_\Omega|\nabla\omega_n|^2+o_n(1).
		\end{split}
		\end{equation}
		Passing to a subsequence, we may assume that
		\begin{equation*}
		\int_\Omega|\nabla\omega_n|^2=\sigma+o_n(1).
		\end{equation*}
		Letting $n\to+\infty$ in \eqref{A.2}, we have
		$$
		c_\rho \leq I_\nu(u) \leq I_\nu(u)+\frac{1}{N} \sigma=\lim _{n \rightarrow \infty} I_\nu\left(\tilde{u}_n\right)=c_\rho,
		$$
		where we have used the fact that $|\nabla u|_2<\rho$ and which implies  that $\sigma=0$. Hence,   we obtain
		$$
		\tilde{u}_n \rightarrow u \text { strongly in } H_0^1(\Omega) .
		$$
		
		Since $I_\nu(u)=c_\rho<0$, we have $u \neq 0$. Then, by a similar argument as used to  the proof of  Theorem 2 in \cite{Deng}, we can get that $u>0$ and $u \in C^2(\Omega)$. This completes the proof.
	\end{proof}
	
	\begin{proof}[\textbf{The proof on $c_\eta$ of  Theorem \ref{buth1}:}]
		By Lemma \ref{Lemma3},  we have $-\infty<c_\eta<0$. Take a minimizing sequence $\left\{u_n\right\} \subset \eta$ for $c_\eta$. Then $I_\nu^{\prime}\left(u_n\right)=0$, $I_\nu\left(u_n\right)\to c_\eta$. By Lemma \ref{Lemma2}, we can see that $\left\{u_n\right\}$ is bounded in $H_0^1(\Omega)$. Hence, we may assume that
		$$
		\begin{aligned}
		&u_n \rightharpoonup u \text { weakly in } H_0^1(\Omega),\\
		& u_n \rightharpoonup u \text { weakly in } L^{2^*}(\Omega), \\
		& u_n \rightarrow u \text { strongly in } L^m(\Omega) \text { for } 1\leq m<2^*, \\
		& u_n \rightarrow u \text { almost everywhere in } \Omega .
		\end{aligned}
		$$
		
		So we have that  $I_\nu^{\prime}(u)=0$ and $u \geq 0$. Letting  $\omega_n=u_n-u$, one gets that
		\begin{equation}\label{A.3}
		\begin{split}
		\int_\Omega|\nabla\omega_n|^2=\mu\left|\omega_n^{+}\right|_{2^*}^{2^*}+o_n(1) . \\
		I_\nu\left(u_n\right)=I_\nu(u)+\frac{1}{N} \int_{\Omega}\left|\nabla \omega_n\right|^2+o_n(1) .
		\end{split}
		\end{equation}

		Passing to a subsequence, we may assume that
		$$
		\int_{\Omega}\left|\nabla \omega_n\right|^2=\sigma+o_n(1) .
		$$
		
		Letting $n \rightarrow+\infty$ in \eqref{A.3} and by $I_\nu^\prime(u)=0$, we have that
		$$
		c_\eta \leq I_\nu(u) \leq I_\nu(u)+\frac{1}{N} \sigma=\lim _{n \rightarrow \infty} I_\nu\left(u_n\right)=c_\eta ,
		$$
		which implies that $\sigma=0$. Hence we obtain
		$$
		u_n \rightarrow u \text { strongly in } H_0^1(\Omega) .
		$$
		
		Since $I_\nu(u)=c_\eta<0$,  we have that $u \neq 0$. By a similar argument as used in the proof of Theorem 1.2 in \cite{Deng},  we can show that $u>0$ and $u \in C^2(\Omega)$. We  complete the proof.
	\end{proof}

	\begin{proof}[\textbf{The proof of   Theorem \ref{buth3}:}]
		Assume that $c_\rho$ is attained by $u_0\in A:=\{u\in H^1_0(\Omega):|\nabla u|_2<\rho\}.$ Then $u_0\in \eta,$ which implies that
		\begin{equation}\label{hg3}
		c_\rho=I_\nu(u_0)\geq c_\eta.
		\end{equation}

		On the other hand, assuming  that $c_\eta$ is attained by some $u_1\in \eta,$ we have that $I_\nu(u_1)=c_\eta<0$, $I_\nu^\prime(u_1)=0$ and $u_1>0$ in $\Omega$.
		We consider the function $g(t):=I_\nu(tu_1), t>0.$ Using the Lemmas \ref{Lemma1}, \ref{lmp3} and the facts that $ g(t)<0$ for $t>0$ small enough and $\lim\limits_{t\to +\infty}g(t)=-\infty $,  we can see that $g(t)$ has only two extreme points $t_1, t_2\in (0, +\infty)$ with $t_1<t_2$.  At the same time, we also have that $t_1$ and $t_2$  are the  local minimum point  and the   maximum point of $g(t)$ respectively,  and  $g(t_2)>0$ and \begin{equation}\label{hg1}
		g(t_1)<g(t)<0~ \hbox{for~ any}~  t\in (0, t_1).
		\end{equation}
		Since $I_\nu(u_1)=c_\eta<0$ and $I_\nu^\prime(u_1)=0$, we see that $g(1)<0$  and $g^\prime(1)=0,$ which tells us that
		$1$ is a extreme point of $g(t)$ and $g(1)<0$. So $t_1=1$. According to \eqref{hg1}, we have that
		\begin{equation}\label{hg2}
		g(t)<0~ \hbox{for~ any}~  t\in (0, 1].
		\end{equation}
		We claim that $|\nabla u_1|_2 <\rho.$ In fact, if  $|\nabla u_1|_2 \geq \rho,$ we can find a $t_3\in (0, 1]$ such that $t _3|\nabla u_1|_2 =\rho,$ which, together with Lemma \ref{Lemma1}, implies that $g(t_3)=I_\nu(t_3u_1)\geq \delta>0,$ contradicting to \eqref{hg2}. Thus, $|\nabla u_1|_2 <\rho,$ which implies that $u_1\in A$ and
		\begin{equation}\label{hg4}
		c_\rho\leq I_\nu(u_1)=c_\eta.
		\end{equation}
		
		According to \eqref{hg3} and \eqref{hg4}, we obtain that $c_\rho=c_\eta.$
		We complete the proof.
	\end{proof}

	\section{Some important inequalities and energy estimates }

	In this Section, we will construct a $\gamma_0(t)\in \Gamma $  such that $c_M\leq \sup\limits_{t\in [0,1]}I_\nu(\gamma_0(t))<c_\eta+\frac{1}{N}\mu^{-\frac{N-2}{2}}S^\frac{N}{2}$. Without loss of generality, we may assume that $0\in \Omega$ and there exists a $r_0>0$  such that $B(0, 4r_0):=\{x\in \R^N: |x|\leq 4r_0\}\subset \Omega.$

	It is well-known  (see \cite{edm,guy,van})  that the following problem
	\begin{gather*}
	\begin{cases}
	-\Delta u=\left|u\right|^{2^\ast-2}u ,& x\in\mathbb{R}^N,\\
	\quad\:\,\,u>0,&\\
	\:\,u(0)=\max\limits_{x\in\mathbb{R}^N} u(x),&\\
	\end{cases}
	\end{gather*}
	up to dilations,  has a unique solution $\widetilde{u}(x)$ and
	\begin{equation*}
	{\widetilde{u}}(x)=\left[N(N-2)\right]^{\frac{N-2}{4}}\frac{1}{{(1+\left|x\right|^2)}^{\frac{N-2}{2}}}.
	\end{equation*}
	It is easy to check that
	\begin{equation*}
	u_{\frac{1}{n}}(x)=\left[N(N-2)\right]^{\frac{N-2}{4}}
	\left(\frac{\frac{1}{n}}{\frac{1}{n^2}+\left|x\right|^2}\right)^{\frac{N-2}{2}}=\left[N(N-2)\right]^{\frac{N-2}{4}}
	\left(\frac{n}{1+n^2\left|x\right|^2}\right)^{\frac{N-2}{2}}\\
	\end{equation*}
	is a minimizer for $S$.
	We define
	\begin{equation}\label{Un}
	U_n(x)=\left\{
	\begin{array}{ll}
	u_{\frac{1}{n}}(x), &0\leq |x|\leq r_0;\\
	\frac{1}{r_0}(2r_0-|x|)u_{\frac{1}{n}}(x), &r_0\leq |x|\leq 2r_0;\\
	0,  &|x|\geq 2r_0.
	\end{array}
	\right.
	\end{equation}

	By a direct computation, we obtain that,  as $n \rightarrow \infty$,
	$$
	\begin{gathered}
	\left\|\nabla U_n\right\|_2^2=\int_{\mathbb{R}^N}\left|\nabla U_n\right|^2 \mathrm{~d} x=S^{\frac{N}{2}}+O\left(\frac{1}{n^{N-2}}\right), \\
	\left\|U_n\right\|_{2^*}^{2^*}=S^{\frac{N}{2}}+O\left(\frac{1}{n^N}\right),
	\end{gathered}
	$$
	and

	$$
	\left\|U_n\right\|_p ^p = \begin{cases}O\left(\frac{1}{\left.n^{\min \left\{\frac{(N-2) p }{2}, N-\frac{(N-2) p }{2}\right\}}\right.}\right), & \text { if } p
	\in [1, 2^*)\setminus\{\frac{N}{N-2}\};
	\\ O\left(\frac{\ln n}{n^{\frac{N}{2}}}\right), & \text { if } p =\frac{N}{N-2}.\end{cases}
	$$
	
	In particular,
	$$
	\left\|U_n\right\|_2^2= \begin{cases}O\left(\frac{1}{n}\right), & \text { if } N=3 ; \\ O\left(\frac{\ln n}{n^2}\right), & \text { if } N=4; \quad \text { as } n \rightarrow \infty . \\ O\left(\frac{1}{n^2}\right), & \text { if } N \geq 5 ;\end{cases}
	$$
	and, for any $p \in [1, 2^*)$,
	$$
	\left\|U_n\right\|_p ^p \leq C\frac{1}{\left.n^{\min \left\{\frac{(N-2) p }{2}, N-\frac{(N-2) p }{2}\right\}}\right.}\ln n.$$
	
	Since $u_0$ is a positive ground state solution of \eqref{1.1} and $B(0, 2r_0)\subset B(0, 4r_0)\subset \Omega$, there exist $0<L_1\leq L_2<+\infty$ such that
	\begin{equation}\label{uv}
	L_1\leq u_0\leq L_2,~x\in B(0, 2r_0).
	\end{equation}

	\begin{lemma}\label{lm4.1}
		Assume that $0<C_1<C_2<+\infty$. Then, for any $\epsilon>0$,  there exits $  A_1>0$, dependent of $\epsilon$,  such that
		$$g(t,y)\leq   y^{2+\epsilon}+A_1y^2, ~~(t,y)\in [C_1, C_2]\times \R^+,$$
		where $g(t,y):=(t+y)^2\ln (t+y)^2-t^2\ln t^2-2ty(\ln t^2+1).$
	\end{lemma}
	\begin{proof}
		For any $t\in [C_1, C_2]$, we have that, as $y\to +\infty,$
		$$ \frac{|g(t,y)|}{y^{2+\epsilon}}\leq \frac{C(C_2+y)^{2-\epsilon}+C(C_2+y)^{2+\frac{\epsilon}{2}}+C+Cy}{y^{2+\epsilon}} \to 0 .$$
		So there exists $X_0>0$ such that
		\begin{equation}\label{lme1}
		g(t,y)\leq y^{2+\epsilon},~\forall ~y\geq X_0 ~\hbox{and}~t\in   [C_1, C_2].
		\end{equation}
		
		On the other hand, we have that, for any $t\in   [C_1, C_2]$, as $y\to 0^+$,
		\begin{equation*}
		\begin{split}
		\frac{g(t,y)}{y^2}&=\frac{2(t+y)^2(\ln t+\ln (1+\frac{y}{t}) -t^2\ln t^2-2ty(\ln t^2+1)}{y^2}\\
		&=\frac{2(t+y)^2(\ln t+\frac{y}{t}-\frac{y^2}{2t^2}+O(\frac{y^3}{t^3})) -t^2\ln t^2-2ty(\ln t^2+1)}{y^2}\\
		&=\ln t^2 +\frac{2(t+y)^2(\frac{y}{t}-\frac{y^2}{2t^2}+O(\frac{y^3}{t^3}))  -2ty}{y^2} \\
		&=\Big[\ln t^2+3)+ O(\frac{y}{t}+\frac{y^2}{t^2}+ \frac{y^3}{t^3})\Big]  \\
		&\leq \big[\ln C_2^2+3)+ O(y)+O(y^2)+ O(y^3)\big]
		\end{split}
		\end{equation*}
		
		which implies that there exists $X_1>0$ such that
		\begin{equation}\label{lme2}
		g(t,y)\leq (4+\ln C_2^2)y^2,~0<y\leq X_1~\hbox{and}~t\in   [C_1, C_2].
		\end{equation}
		At the same time, it is easy to see that $\sup\limits_{(t,y)\in \Lambda} g(t,y)$
		is  attained, where $ \Lambda ={[C_1, C_2]\times [X_1, X_0]}$. So there exists a positive constant  $A_1>4+\ln C_2^2$ such that when $(t,y)\in \Lambda,$
		$$g(t,y)\leq \sup\limits_{(t,y)\in \Lambda} g(t,y)\leq A_1X_1^2\leq A_1y^2,$$
		which, together with \eqref{lme1} and \eqref{lme2}, implies that
		$$g(t,y)\leq y^{2+\epsilon}+A_1y^2$$
		for any $(t, y)\in [C_1, C_2]\times \R^+.$
		We complete the proof.
	\end{proof}

	\begin{lemma}\label{lm4.2}Assume that   $0<C_1<C_2<+\infty$. Then there exists a constant  $A_2,\hat{A}_2>0$ such that, for $p\in (3, +\infty)$,
		$$f(p,t,y)
		\geq \frac{1}{2}pC_1 y^{p-1}-A_2y^2, ~~(t,y)\in [C_1, C_2]\times \R^+,$$
		and, for $p\in (2,+\infty)$,
		$$|f(p,t,y)| \leq \frac{1}{2}p^2C_2^{p-2}y^2+\hat{A}_2C_2y^{p-1}, ~~(t,y)\in [C_1, C_2]\times \R^+,$$
		where $f(p,t,y) =(t+y)^p -t^p -y^p-pt^{p-1}y.$
	\end{lemma}
	\begin{proof}
		Set $\tilde{f}(y) =(1+y)^p -1 -y^p-py.$

		\textbf{Case 1:}~$p\in (3, +\infty)$
		
		It is easy to see that, as $y\to +\infty,$
		\begin{equation}\label{lme2.1}
		\begin{split}
		\frac{\tilde{f}(y)}{y^{p-1}}&=\frac{(1+y)^p   -y^p}{y^{p-1}}+o(1)=y\Big[(1+\frac{1}{y})^p-1\Big]+o(1)=p+o(1),
		\end{split}
		\end{equation}
		which implies that
		there exists $X_2>0$ such that
		\begin{equation}\label{lme2.2}
		\tilde{f}(y)\geq \frac{1}{2}p y^{p-1}~\hbox{for~any}~y\in   [X_2, +\infty).
		\end{equation}
		
		On the other hand, we obtain that, as $y\to 0^+$,
		\begin{equation}\label{lme2.3}
		\begin{split}
		\frac{\tilde{f}(y)-\frac{1}{2}p  y^{p-1}}{y^2}&=\frac{ (1+y)^p -1  -py}{y^2}+o(1)\\
		&= \frac{1}{2}\frac{\partial^2 (1+y)^p}{\partial y^2}\Big|_{y=0}+o(1)\\
		&=\frac{1}{2}p(p-1)+o(1),
		\end{split}
		\end{equation}
		which implies that there exists $X_3>0$ such that, for any $y\in  (0, X_3]$,
		$$\tilde{f}(y)-\frac{1}{2}p  y^{p-1}\geq \frac{1}{4}p(p-1)y^2 .$$ That is,
		\begin{equation}\label{lme12.4}
		\tilde{f}(y)\geq \frac{1}{2}p  y^{p-1}+\frac{1}{4}p(p-1)y^2, y\in  (0, X_3].
		\end{equation}
		Since $\tilde{f}(y)-\frac{1}{2}p  y^{p-1}$ is continuous in $ [X_3, X_2]$, we can find a constant $A_2>0$ such that,
		for any $y\in [X_3, X_2]$,
		$$\tilde{f}(y)-\frac{1}{2}p  y^{p-1}\geq \inf\limits_{y\in [X_3, X_2]}[\tilde{f}(y)-\frac{1}{2}p  y^{p-1}]\geq-A_2C_2^{2-p}X_3^2\geq -A_2C_2^{2-p}y^2,$$
		which, together with \eqref{lme2.2} and \eqref{lme12.4}, implies that
		$$\tilde{f}(y)\geq \frac{1}{2}p  y^{p-1}-A_2C_2^{2-p}y^2$$
		for any $y\in (0, +\infty).$
		It is easy to see that $f(p, t,y)=t^p\tilde{f}(\frac{y}{t}).$
		So for $(t,y)\in [C_1, C_2]\times \R^+$,
		$$f(p, t,y)=t^p\tilde{f}(\frac{y}{t})\geq \frac{1}{2}p  ty^{p-1}-A_2C_2^{2-p}t^{p-2}y^2\geq \frac{1}{2}p  C_1y^{p-1}-A_2y^2.$$
		
		\textbf{Case 2:}~$p\in (2, +\infty)$
		
		One can see that,   as $y\to 0^+$,
		\begin{equation}
		\begin{split}
		\Big|\frac{\tilde{f}(y)}{y^2}\Big|&=\Big|\frac{ (1+y)^p -1  -py}{y^2}+o(1)\Big|= \frac{1}{2}\frac{\partial^2 (1+y)^p}{\partial y^2}\Big|_{y=0}+o(1)=\frac{1}{2}p(p-1)+o(1).
		\end{split}
		\end{equation}
		So there exists a $\hat{X}_2>0$ such that
		\begin{equation}\label{bulme12.4}
		|\tilde{f}(y)| \leq \frac{1}{2}p^2y^2, y\in  (0, \hat{X}_2].
		\end{equation}

		On the other hand, we have that, as $y\to +\infty,$
		\begin{equation}
		\begin{split}
		\frac{|\tilde{f}(y)|}{y^{p-1}}&=\Big|\frac{(1+y)^p   -y^p}{y^{p-1}}+o(1)\Big|=y\Big[(1+\frac{1}{y})^p-1\Big]+o(1)=p+o(1),
		\end{split}
		\end{equation}
		which means  that
		there exists $\hat{X}_3>0$ such that
		\begin{equation}\label{bulme2.2}
		|\tilde{f}(y)|\leq 2p y^{p-1}~\hbox{for~any}~y\in   [\hat{X}_3, +\infty).
		\end{equation}
		Since $|\tilde{f}(y)|$ is continuous in $[\hat{X}_2, \hat{X}_3],$ we have that, for any $y\in  [\hat{X}_2, \hat{X}_3],$
		$$|\tilde{f}(y)|\leq \max\limits_{y\in[\hat{X}_2, \hat{X}_3]} |\tilde{f}(y)|\leq C\leq C\min\limits_{y\in[\hat{X}_2, \hat{X}_3]}y^{p-1}\leq Cy^{p-1},$$
		which, together with \eqref{bulme12.4} and \eqref{bulme2.2}, implies that there exists a $\hat{A}_2>0$ large enough such that
		$$|\tilde{f}(y)|\leq \frac{1}{2}p^2y^2+\hat{A}_2y^{p-1}, ~y\in \R^+.$$
		So
		$$|f(p, t,y)|=|t^p\tilde{f}(\frac{y}{t})|\leq \frac{1}{2}p^2t^{p-2}y^2+\hat{A}_2ty^{p-1}\leq   \frac{1}{2}p^2C_2^{p-2}y^2+\hat{A}_2C_2y^{p-1},~y\in \R^+.$$

		This completes the proof.
	\end{proof}

	\begin{lemma}\label{lm4.3}Assume that $m\in (2, +\infty)$ and  $0<C_1<C_2<+\infty$. Then there exists a constant  $A_3>0$ large enough  such that
		$$2 y^m+A_3y^2\geq f_1(m, t,y)\geq \frac{1}{2}   y^m-A_3y^2, ~~(t,y)\in [C_1, C_2]\times \R^+,$$
		where $f_1(m, t,y) =(t+y)^m -t^m -mt^{m-1}y.$
	\end{lemma}
	\begin{proof}
		Set $\tilde{f}_1(y) =(1+y)^m -1 -m y.$ Then
		it is easy to see that
		\begin{equation}\label{lme3.1}
		\lim\limits_{y\to+\infty}\frac{\tilde{f}_1(y)}{y^{m}}=1,
		\end{equation}
		which implies that
		there exists $X_4>0$ such that
		\begin{equation}\label{lme13.2}
		2y^m\geq \tilde{f}_1(y)\geq \frac{1}{2}  y^m~\hbox{for~any}~y\in   [X_4, +\infty).
		\end{equation}
		
		On the other hand, we obtain that, as $y\to 0^+$,
		\begin{equation}\label{lme2.3}
		\begin{split}
		\frac{\tilde{f}_1(y)-\frac{1}{2} y^m}{y^2}&=\frac{ (1+y)^m -1  -my}{y^2}+o(1)\\
		&= \frac{1}{2}\frac{\partial^2 (1+y)^m}{\partial y^2}\Big|_{y=0}+o(1)\\
		&=\frac{1}{2}m(m-1)+o(1),
		\end{split}
		\end{equation}
		which implies that there exists $X_5>0$ such that, for any $y\in  (0, X_5]$,
		$$\frac{1}{2}m^2y^2\geq \tilde{f}_1(y)-\frac{1}{2}y^m\geq \frac{1}{4}m(m-1)y^2 .$$
		That is,
		\begin{equation}\label{lme2.4}
		Cy^2\geq \frac{1}{2}m^2y^2+\frac{1}{2}y^m\geq \tilde{f}_1(y)\geq \frac{1}{2}   y^m+\frac{1}{4}m(m-1)y^2, y\in  (0, X_5],
		\end{equation}
		since $m>2.$
		Since $\tilde{f}_1(y)$ and  $\tilde{f}_1(y)-\frac{1}{2}   y^m$ are  continuous in $ [X_5, X_4]$, we can find a constant $A_3>0$ large enough such that,
		for any $y\in [X_5, X_4]$,
		$$\tilde{f}_1(y)-\frac{1}{2}   y^m\geq \inf\limits_{y\in [X_5, X_4]}[\tilde{f}_1(y)-\frac{1}{2}   y^m]\geq-A_3C_2^{2-m}X_5^2\geq -A_3C_2^{2-m}y^2,$$
		and
		$$\tilde{f}_1(y)\leq \max\limits_{y\in [X_5, X_4]}[\tilde{f}_1(y)]\leq A_3C_2^{2-m}X_5\leq  A_3C_2^{2-m}y^2,$$
		which, together with \eqref{lme13.2} and \eqref{lme2.4}, implies that
		$$2 y^m+A_3C_2^{2-m}y^2\geq \tilde{f}_1(y)\geq \frac{1}{2}   y^m-A_3C_2^{2-m}y^2$$
		for any $y\in (0, +\infty).$
		It is easy to see that $f_1(m, t, y)=t^m\tilde{f}_1(\frac{y}{t}).$
		So for $(t,y)\in [C_1, C_2]\times \R^+$,
		$$f_1(m, t, y)=t^m\tilde{f}_1(\frac{y}{t})\geq \frac{1}{2}   y^m-A_3C_2^{2-m}t^{m-2}y^2\geq \frac{1}{2}  y^m-A_3y^2,$$
		and
		$$f_1(m, t, y)=t^m\tilde{f}_1(\frac{y}{t})\leq 2   y^m+A_3C_2^{2-m}t^{m-2}y^2\leq 2  y^m+A_3y^2.$$
		This completes the proof.
	\end{proof}

	\begin{corollary}\label{c1}
		For any $\epsilon>0$, there exist $B_1, B_2, B_3>0$ such that, for any $t_n\in (0, +\infty)$,
		$$g(u_0, t_nU_n)\leq (t_nU_n)^{2+\epsilon}+B_1(t_nU_n)^2,~\forall x\in \Omega,$$
		$$f(2^*, u_0, t_nU_n)\geq \frac{2^*}{2}L_1(t_nU_n)^{2^*-1}-B_2(t_nU_n)^2, ~\forall x\in \Omega,N=3,4,5,$$
		$$|f(2^*, u_0, t_nU_n)|\leq \frac{(2^*)^2}{2}L^{2^*-2}_2(t_nU_n)^2+B_2L_2(t_nU_n)^{2^*-1}, ~\forall x\in \Omega, N\geq 3,$$
		and
		$$2 (t_nU_n)^q+B_3(t_nU_n)^2\geq f_1(q,u_0, t_nU_n)\geq \frac{1}{2}(t_nU_n)^q-B_3(t_nU_n)^2, ~~\forall  x\in \Omega,$$
		where $L_1$ and $L_2$ are  given in \eqref{uv}, and  $g(t,y), f(p,t, y)$ and $f_1(m,t, y)$ are defined in Lemmas \ref{lm4.1}, \ref{lm4.2} and \ref{lm4.3}, respectively.
	\end{corollary}
	\begin{proof}
		When $x\in \Omega\setminus B(0, 2r_0)$, we have that $U_n(x)=0$. So it is easy to see that $g(u_0, t_nU_n)=f(2^*, u_0, t_nU_n)=f_1(q,u_0, t_nU_n)=0, x\in \Omega\setminus B(0, 2r_0),$
		which implies that the conclusions hold for $x\in \Omega\setminus B(0, 2r_0)$.
		
		Let $t=u_0$ and $y=t_nU_n$.
		If $x\in  B(0, 2r_0)$, then $t=u_0\in [L_1, L_2]$. Thus, by Lemmas \ref{lm4.1}, \ref{lm4.2} and \ref{lm4.3}, one can see that the conclusions hold for $x\in   B(0, 2r_0).$
		We complete the proof.
	\end{proof}

	\begin{lemma}\label{lm5.1}Assume that $\nu\in\R$ and $q\in   (2, 2^*)$. Then there exist $N_0>0$ and $ T>4\rho/S^\frac{N}{4}$, dependent of $N_0$,  such that when $n\geq N_0$,
		$$I_\nu(u_0+TU_n)<I_\nu(u_0)<0.$$
	\end{lemma}
	\begin{proof}
		By direct computations, we have that, for $n$ large enough,
		\begin{equation*}
		\begin{split}
		&I_\nu(u_0+tU_n)\\
		&=\frac{1}{2}\int |\nabla (u_0+tU_n)|^2-\frac{\mu}{2^*}\int |u_0+tU_n|^{2^*}-\frac{\nu}{q}\int |u_0+tU_n|^q-\frac{\lambda}{2}\int (u_0+tU_n)^2\\
		&
		-\frac{\theta}{2}\int(u_0+tU_n)^2(\ln(u_0+tU_n)^2-1)\\
		&\leq \int|\nabla u_0|^2 +\int|\nabla U_n|^2t^2 -\frac{\mu}{2^*}\int|U_n|^{2^*} t^{2^*}
		+C\int(u_0+tU_n)^q +C\int(u_0+tU_n)^{2-\epsilon}\\
		&\leq C + \int|\nabla U_n|^2 t^2-\frac{\mu}{2^*}\int|U_n|^{2^*} t^{2^*}
		+C\int U_n ^q t^q+C\int U_n ^{2-\epsilon} t^{2-\epsilon}\\
		&\leq C+(2S^\frac{N}{2})t^2+Ct^q+Ct^{2-\epsilon}-\frac{\mu}{2^*}(\frac{1}{2}S^\frac{N}{2})t^{2^*}\\
		&\to -\infty,~\hbox{as}~t\to +\infty,
		\end{split}
		\end{equation*}
		where $0<\epsilon<1$ and which implies that there exist $N_0>0$ and $ T>4\rho/S^\frac{N}{4}$, dependent of $N_0$, such that when $n\geq N_0$,
		$$I_\nu(u_0+TU_n)<I_\nu(u_0)<0.$$
		We complete the proof.
		
	\end{proof}

	\begin{lemma}\label{lm5.2}
		For any $n$ large enough, there exists a $t_n\in (0, T)$ such that
		$$I_\nu(u_0+t_nU_n)=\max\limits_{t\in [0, T]}I_\nu(u_0+tU_n)\geq \delta$$
		and
		$$0<  \inf\limits_{n}t_n\leq t_n\leq T.$$
	\end{lemma}
	
	\begin{proof}
		It follows from  $T>4 \rho /S^\frac{N}{4}$ and $||u_0||<  \rho $ that, for any $n$ large enough,
		$$||u_0+TU_n||\geq T||U_n||-||u_0||>\rho,$$
		which, together with $||u_0||<  \rho $, implies that there exists $t_{n,0}\in (0, T]$ such that
		$||u_0+t_{n,0}U_n||=\rho$
		and
		$$I_\nu(u_0+t_{n,0}U_n)\geq \delta>0>\max\{I_\nu(u_0), I_\nu(u_0+TU_n)\}.$$
		So   there exists $t_n\in (0, T)$ such that
		$$I_\nu(u_0+t_nU_n)=\max\limits_{t\in [0, T]}I_\nu(u_0+tU_n)\geq \delta.$$
		
		We have known that $t_n\leq T$. So we only need to prove that $\inf\limits_{n}t_n> 0$. Assuming that $t_n\to0$, then
		$||u_0+t_nU_n||^2 \to ||u_0||^2$ as $n\to +\infty,$
		which implies that
		$$0<\delta\leq \lim\limits_{n\to+\infty}I_\nu(u_0+t_nU_n)=I_\nu(u_0)<0,$$
		leading to a contradiction. Thus $\inf\limits_{n}t_n> 0$.
		This completes the proof.
	\end{proof}
	
	Set $$\gamma_0(t):=u_0+tTU_n, t\in [0,1].$$ It is easy to see that $\gamma_0(t)\in \Gamma $ and $c_M\leq \sup\limits_{t\in [0,1]}I_\nu(\gamma_0(t))\leq I_\nu(u_0+t_nU_n).$

	\begin{lemma}\label{lm5.3}Assume that $3\leq N\leq 5,$   $\nu\leq 0$ and $ q\in(2,2^*-1)$. Then there exists $N_1\geq N_0$ such that when $n\geq N_1,$
		$$c_M\leq I_\nu(u_0+t_nU_n)<c_\eta+\frac{1}{N}\mu^{-\frac{N-2}{2}}S^{\frac{N}{2}}.$$
	\end{lemma}
	\begin{proof}For any  $\nu\leq 0$ and $q\in(2,2^*-1)$,  direct computations implies  that, for $n$ large enough,
		\begin{equation}\label{lme5.32}
		\begin{aligned}
		&I_\nu(u_0+t_nU_n)\\
		&=I_\nu(u_0)+\frac{1}{2}\int|\nabla U_n|^2 t_n^2-\frac{\mu}{2^*}\int(|u_0+t_nU_n|^{2^*}-u_0^{2^*}-2^*u_0^{2^*-1}t_nU_n)\\
		&-\frac{\nu}{q}\int(|u_0+t_nU_n|^{q}-u_0^{q}-qu_0^{q-1}t_nU_n)\\
		&-\frac{\lambda-\theta}{2}\int U_n^2t_n^2-\frac{\theta}{2}\int\Big[(u_0+t_nU_n)^2\ln(u_0+t_nU_n)^2-u_0^2\ln u_0^2-2u_0t_nU_n(\ln u_0^2+1)\Big]\\
		&=c_\eta  +\frac{1}{2}\int|\nabla U_n|^2 t_n^2-\frac{\mu}{2^*}\int|t_nU_n|^{2^*}-\frac{\mu}{2^*}\int f(2^*, u_0, t_nU_n)\\
		&-\frac{\nu}{q}\int f_1(q, u_0, t_nU_n)-\frac{\lambda-\theta}{2}\int U_n^2t_n^2-\frac{\theta}{2}\int g(u_0, t_nU_n)\\
		&\leq c_\eta  +\frac{1}{2}S^{\frac{N}{2}} t_n^2-\frac{\mu}{2^*}S^{\frac{N}{2}}t_n^{2^*}+O\Big(\frac{1}{n^{N-2}}\Big)-\frac{L_1\mu}{2}\int |t_nU_n|^{2^*-1}+C\int |t_nU_n|^{q}+C\int |t_nU_n|^{2}\\
		&=c_\eta  +\frac{1}{N}\mu^{-\frac{N-2}{2}}S^{\frac{N}{2}} +O\Big(\frac{1}{n^{N-2}}\Big)-\frac{L_1\mu}{2}\frac{1}{n^{\frac{N-2}{2}}}
		+O\Big(\frac{\ln n}{n^{\min\{\frac{(N-2)q}{2}, N-\frac{(N-2)q}{2}\}}}\Big)
		+C|U_n|_2^2\\
		&<c_\eta  +\frac{1}{N}\mu^{-\frac{N-2}{N}}S^{\frac{N}{2}},
		\end{aligned}
		\end{equation}
		where $\epsilon=q-2$   and we have used the Corollary \ref{c1} and the facts that $\min\{\frac{(N-2)q}{2}, N-\frac{(N-2)q}{2}\}>\frac{N-2}{2}$ and
		$$
		\left\|U_n\right\|_2^2= \begin{cases}O\left(\frac{1}{n}\right), & \text { if } N=3 ; \\ O\left(\frac{\ln n}{n^2}\right), & \text { if } N=4; \quad \text { as } n \rightarrow \infty . \\ O\left(\frac{1}{n^2}\right), & \text { if } N \geq 5 ;\end{cases}
		$$
		We complete the proof.
	\end{proof}
	\begin{lemma}\label{lm5.4} Assume that $N\geq 3.$  If   $\nu>0$ and $ q\in (2, 2^*)$,  then there exists $N_2\geq N_0$ such that when $n\geq N_2,$
		$$c_M\leq I_\nu(u_0+t_nU_n)<c_\eta +\frac{1}{N}\mu^{-\frac{N-2}{2}}S^{\frac{N}{2}}.$$
	\end{lemma}
	\begin{proof}For any  $\nu>0, N=3,4, 5,$ and $q\in(2,2^*)$,   we have that, for $n$ large enough,
		\begin{equation}\label{lme15.32}
		\begin{split}
		&I_\nu(u_0+t_nU_n)\\
		&=I_\nu(u_0)+\frac{1}{2}\int|\nabla U_n|^2 t_n^2-\frac{\mu}{2^*}\int(|u_0+t_nU_n|^{2^*}-u_0^{2^*}-2^*u_0^{2^*-1}t_nU_n)\\
		&-\frac{\nu}{q}\int(|u_0+t_nU_n|^{q}-u_0^{q}-qu_0^{q-1}t_nU_n)\\
		&-\frac{\lambda-\theta}{2}\int U_n^2t_n^2-\frac{\theta}{2}\int \Big[(u_0+t_nU_n)^2\ln(u_0+t_nU_n)^2-u_0^2\ln u_0^2-2u_0t_nU_n(\ln u_0^2+1) \Big]\\
		&=c_\eta  +\frac{1}{2}\int|\nabla U_n|^2 t_n^2-\frac{\mu}{2^*}\int|t_nU_n|^{2^*}-\frac{\mu}{2^*}\int f(2^*, u_0, t_n U_n)\\
		&-\frac{\nu}{q}\int f_1(q, u_0, t_nU_n)-\frac{\lambda-\theta}{2}\int U_n^2t_n^2-\frac{\theta}{2}\int g(u_0, t_nU_n)\\
		&\leq \Big[c_\eta  +\frac{1}{2}S^{\frac{N}{2}} t_n^2-\frac{\mu}{2^*}S^{\frac{N}{2}}t_n^{2^*}+O\Big(\frac{1}{n^{N-2}}\Big)-\frac{L_1\mu}{2}\int |t_nU_n|^{2^*-1}\\
		&-C\nu\int |t_nU_n|^q+C\int |t_nU_n|^{2}-\frac{\theta}{2}\int(t_nU_n)^{2+\epsilon_1}\Big]\\
		&\leq \Big[c_\eta  +\frac{1}{N}\mu^{-\frac{N-2}{2}}S^{\frac{N}{2}} +O\Big(\frac{1}{n^{N-2}}\Big)-C\frac{1}{n^{\frac{N-2}{2}}}+O \Big(\frac{1}{n^{\min\{\frac{(N-2)(2+\epsilon_1)}{2}, 2-\frac{(N-2)\epsilon_1}{2}\}}} \Big)+C|U_n|_2^2\Big]\\
		&<c_\eta  +\frac{1}{N}\mu^{-\frac{N-2}{2}}S^{\frac{N}{2}},\\
		\end{split}
		\end{equation}
		where $\epsilon_1\in (0,1)$ with $\min\{\frac{(N-2)(2+\epsilon_1)}{2}, 2-\frac{(N-2)\epsilon_1}{2}\}>\frac{N-2}{2}$
		and we have used the Corollary \ref{c1} and the fact that
		$$
		\left\|U_n\right\|_2^2= \begin{cases}O\left(\frac{1}{n}\right), & \text { if } N=3 ; \\ O\left(\frac{\ln n}{n^2}\right), & \text { if } N=4; \quad \text { as } n \rightarrow \infty . \\ O\left(\frac{1}{n^2}\right), & \text { if } N \geq 5 .\end{cases}
		$$
		Similar to \eqref{lme15.32}, when   $\nu>0, N\geq 6,$ and $q\in(2,2^*)$,   we obtain  that, for $n$ large enough,
		\begin{align*}
		&I_\nu(u_0+t_nU_n)\\
		&=c_\eta  +\frac{1}{2}\int|\nabla U_n|^2 t_n^2-\frac{\mu}{2^*}\int|t_nU_n|^{2^*}-\frac{\mu}{2^*}\int f(2^*, u_0, t_n U_n)\\
		&-\frac{\nu}{q}\int f_1(q, u_0, t_nU_n)-\frac{\lambda-\theta}{2}\int U_n^2t_n^2-\frac{\theta}{2}\int g(u_0, t_nU_n)\\
		&\leq c_\eta  +\frac{1}{2}S^{\frac{N}{2}} t_n^2-\frac{\mu}{2^*}S^{\frac{N}{2}}t_n^{2^*}+O\Big(\frac{1}{n^{N-2}}\Big)+C\int |t_nU_n|^{2^*-1}
		-C\nu\int |t_nU_n|^q+C\int |t_nU_n|^{2}\\
		&-\frac{\theta}{2}\int(t_nU_n)^{2+\epsilon_1}\\
		&\leq c_\eta  +\frac{1}{N}\mu^{-\frac{N-2}{2}}S^{\frac{N}{2}} +O\Big(\frac{1}{n^{N-2}}\Big)+C\frac{1}{n^{\frac{N-2}{2}}}-C\nu\frac{1}{n^{\min\{\frac{(N-2)q}{2}, N-\frac{(N-2)q}{2}\}}}\\
		&+O\Big(\frac{1}{n^{\min\{\frac{(N-2)(2+\epsilon_1)}{2}, 2-\frac{(N-2)\epsilon_1}{2}\}}}\Big)
		+C|U_n|_2^2\\
		&<c_\eta  +\frac{1}{N}\mu^{-\frac{N-2}{2}}S^{\frac{N}{2}},
		\end{align*}
		where $\epsilon_1\in (0,1)$ with $\frac{N-2}{2}\geq 2>\min\{\frac{(N-2)(2+\epsilon_1)}{2}, 2-\frac{(N-2)\epsilon_1}{2}\}>\min\{\frac{(N-2)q}{2}, N-\frac{(N-2)q}{2}\}$
		and we have used the Corollary \ref{c1} and the fact that $\left\|U_n\right\|_2^2=O\big(\frac{1}{n^2}\big).$
		This complete the proof.
		
	\end{proof}

	\section{The proof of Theorem \ref{th4}}
	
	\begin{proof}[\textbf{The proof of Theorem \ref{th4}:}]
		It follows from the Mountain pass Theorem that there exists a sequence $\{u_n\} \subset H^1_0(\Omega)$ such that
		$I_\nu(u_n)\to c_M$ and $I^\prime_\nu(u_n)\to0$. By Lemma \ref{Lemma2}, we can see that  $\{u_n\} $ is bounded in $H^1_0(\Omega)$.
		So there exists $u\in H^1_{0}(\Omega)$ such that
		$$
		\begin{aligned}
		&u_n \rightharpoonup u \text { weakly in } H_0^1(\Omega),\\
		& u_n \rightharpoonup u \text { weakly in } L^{2^*}(\Omega), \\
		& u_n \rightarrow u \text { strongly in } L^m(\Omega) \text { for } 1\leq m<2^*, \\
		& u_n \rightarrow u \text { almost everywhere in } \Omega .
		\end{aligned}
		$$
		Since $<I_\nu^{'}(u_{n}),\varphi>\to0$ as $n\to\infty$ for any $\varphi\in C^{\infty}_{0}(\Omega)$, $u$ is a weak solution of
		\begin{gather*}
		-\Delta u=\mu\left|u_{+}\right|^{2^\ast-2}u_{+}+\nu\left|u_{+}\right|^{q-2}u_{+}+\lambda u_++\theta u_{+}\log u^2_{+},
		\end{gather*}
		which implies that

		\begin{equation}\label{2.5}
		\begin{array}{ll}
		I_\nu(u)\geq c_\eta.
		\end{array}
		\end{equation}
		Following from the definition of $(PS)_{c_M}$ sequence, we have
		$$
		\int\left|\nabla u_{n}\right|^2-\mu\int\left|(u_{n})_{+}\right|^{2^\ast}-\nu\int\left|(u_{n})_{+}\right|^{q}-\lambda\int (u_{n})^2_{+}-\theta\int (u_{n})^2_{+}\log (u_{n})^2_{+}=o_{n}(1)$$
		and
		$$		\frac{1}{2}\int\left|\nabla u_{n}\right|^2-\frac{\mu}{2^\ast}\int\left|(u_{n})_{+}\right|^{2^\ast}-\frac{\nu}{q}\int\left|(u_{n})_{+}\right|^{q}-\frac{\lambda}{2}\int (u_{n})^2_{+}-\frac{\theta}{2}\int (u_{n})^2_{+}(\log (u_{n})^2_{+}-1)=c_M+o_{n}(1).
		$$
		Set   $v_{n}=u_{n}-u$. Then
		$$
		\int\left|\nabla v_{n}\right|^2-\mu\int\left|(v_{n})_{+}\right|^{2^\ast}=o_{n}(1)$$
		and
		\begin{equation}\label{e4.11}
		I_\nu(u)+\frac{1}{2}\int\left|\nabla v_{n}\right|^2-\frac{\mu}{2^\ast}\int\left|(v_{n})_{+}\right|^{2^\ast}=c_M+o_{n}(1).
		\end{equation}
		Let
		$$
		\ds\int\left|\nabla v_{n}\right|^2\to k, ~ \hbox{as} ~ n\to \infty.
		$$
		So
		$$
		\int\left|(v_{n})_{+}\right|^{2^\ast}\to \frac{k}{\mu}, ~ \hbox{as} ~  n\to\infty.
		$$
		
		By the definition of $S$, we have
		$$      		\left|\nabla u\right|^2_{2} \ge S\left|u\right|^2_{2^\ast},~~\forall u\in H^1_{0}(\Omega)$$
		and
		$$k+o_{n}(1)=\int\left|\nabla v_{n}\right|^2 \ge S\big(\int\left|(v_{n})_{+}\right|^{2^\ast}\big)^{\frac{2}{2^\ast}}=S\mu ^{-\frac{N-2}{N}}k^\frac{N-2}{N}+o_{n}(1).
		$$
		If $k> 0$, then $k\ge \mu ^{-\frac{N-2}{2}}S^\frac{N}{2}$.
		By  $\eqref{2.5}$, we have
		\begin{equation*}
		\begin{split}
		c_M&= \lim\limits_{n\to +\infty}\Big[I_\nu(u)+\frac{1}{2}\int\left|\nabla v_{n}\right|^2-\frac{\mu}{2^\ast}\int\left|(v_{n})_{+}\right|^{2^\ast}\Big]\\
		&\geq c_\eta +\lim\limits_{n\to +\infty}\frac{1}{N}\int\left|\nabla v_{n}\right|^2 \\
		&=c_\eta +\frac{1}{N}k\\
		&\geq c_\eta +\frac{1}{N}\mu ^{-\frac{N-2}{2}}S^\frac{N}{2}\\
		\end{split}
		\end{equation*}
		which contradicts to Lemmas \ref{lm5.3} and \ref{lm5.4}.
		Thus $k=0$.  That is,
		$$\begin{array}{ll}
		u_{n}\rightarrow u,~\hbox{ in} ~ H^1_{0}(\Omega),
		\end{array}$$
		which implies that $I_\nu(u)=c_M$ and $I^\prime_\nu(u)=0$.
		That is, $u$ is a Mountain pass solution of \eqref{1.1}.  Using a similar argument as used to  the proof of  Theorem 2 in \cite{Deng}, we can get that $u>0$ and $u \in C^2(\Omega)$.
		We complete the proof.
	\end{proof}

	
	\bibliographystyle{plainnat}
	
\end{document}